\documentclass[3p,10pt,a4paper,twoside,fleqn,sort&compress]{amsart}
\usepackage{amssymb,amsmath,latexsym,longtable}
\usepackage[varg]{pxfonts}
\usepackage{turnstile,stmaryrd}

\newtheorem{theorem}{Theorem}[section]

\newtheorem{notation}[theorem]{Notation}

\newcommand{\s}{\mathfrak{S}}
\newcommand{\sx}{\mathop{\mathfrak{S}}\limits_{x,y,z}}

\newcommand{\W}{\mathcal{W}}
\newcommand{\T}{\mathcal{T}}

\newcommand{\R}{\mathbb{R}}
\newcommand{\C}{\mathbb{C}}
\newcommand{\N}{\widehat{N}}

\newcommand{\f}{\varphi}
\newcommand{\Id}{\mathrm{Id}}
\newcommand{\M}{(M,\allowbreak\f,\allowbreak\xi,\allowbreak\eta,g)}
\newcommand{\ta}{\theta}
\newcommand{\om}{\omega}
\newcommand{\F}{\mathcal{F}}
\newcommand{\D}{{\rm d}}
\newcommand{\U}{\mathcal{U}}

\newcommand{\n}{\nabla}
\newcommand{\nn}{\nabla^*}

\newcommand{\ie}{i.e. }

\newcommand{\thmref}[1]{The\-o\-rem~\ref{#1}}

\begin{document}

\title{On Canonical-type Connections \\
on Almost Contact Complex Riemannian Manifolds}

\author{Mancho Manev}
\address{Department of Algebra and Geometry, Faculty of Mathematics and
Informatics\\
Paisii Hilendarski University of Plovdiv,
236 Bulgaria Blvd., Plovdiv 4027, Bulgaria
}
\email{mmanev@uni-plovdiv.bg}

\begin{abstract}
We consider a pair of smooth manifolds, which are the counterparts
in the even-dimensional and odd-dimensional cases. They are separately an
almost complex manifold with Norden metric and an almost contact
manifolds with B-metric, respectively. They can be combined as
the so-called almost contact complex Riemannian manifold.
This paper is a survey with additions of
results on differential geometry of canonical-type connections
(i.e. metric connections with torsion satisfying a certain
algebraic identity) on the considered manifolds.
\end{abstract}

\subjclass[2010]{Primary 53C05; Secondary 53C15, 53C50}
\keywords{almost complex manifold, quasi-K\"ahler manifold, almost
contact manifold, Norden metric, B-metric, natural connection,
canonical connection, torsion, conformal transformation}

\maketitle

\section{Introduction}\label{sec:intro}

The geometry of almost Hermitian manifolds $(N,J,h)$ is well
developed. As it is known, P.~Gauduchon  gives in \cite{Gau97} a
unified presentation of a so-called \emph{canonical} class of
(almost) Hermitian connections, considered by
P.~Libermann in \cite{Lib54}. Let us recall, a linear connection $D$
is called \emph{Hermitian} if it preserves the Hermitian metric
$h$ and the almost complex structure $J$, $Dh=DJ=0$. The
\emph{potential} of $D$ (with respect to the Levi-Civita
connection $\n$), denoted by $Q$, is defined by the difference
$D-\n$. The connection $D$ preserves the metric and therefore is
completely determined by its \emph{torsion} $T$.
According to  \cite{Car25,
TV83,Sal89}, the two spaces of all torsions and of all potentials
are isomorphic as $O(n)$ representations and an equivariant
bijection is the following
\begin{gather}
T (x,y,z) = Q(x,y,z) - Q(y,x,z) ,\label{TQ}\\[0pt]
2Q(x,y,z) = T (x,y,z) - T (y,z,x) + T (z,x,y).\label{QT}
\end{gather}
Following E. Cartan \cite{Car25}, there are studied the algebraic types of the torsion
tensor for a metric connection, \ie a linear connection preserving the metric.

On an almost Hermitian manifold, a Hermitian connection
is called \emph{canonical} 
if its torsion $T$ satisfies the following conditions: \cite{Gau97}

1) the component of $T$ satisfying the Bianchi identity and having
the property $T (J \cdot, J \cdot) = T (\cdot, \cdot)$ vanishes;

2) for some real number $t$, it is valid
$(\s T)^+=(1-2t)(\D\Omega)^+(J\cdot,J\cdot,J\cdot)$, where $\s$ denotes the cyclic sum by
three arguments and $(\D\Omega)^+$ is the part of type
$(2,1)+(1,2)$ of the diffe\-ren\-tial  $\D\Omega$ for the K\"ahler
form $\Omega=g(J\cdot,\cdot)$.


According to  \cite{Gau97}, there exists an one-parameter family
$\{\n^t\}_{t\in\R}$ of canonical Hermitian connections $\n^t =
t\n^1 + (1-t)\n^0$, where $\n^0$ and $\n^1$ are the
\emph{Lichnerowicz first and second canonical
connections} \cite{Lih62}, respectively.

An object of our interest is the class of manifolds with
Norden-type metrics.\footnote{Research supported by the project NI13-FMI-002 of the Scientific
Research Fund at the University of Plovdiv}

In comparison, the action of the almost complex structure with
respect to the Hermitian metric (respectively, the Norden metric)
on the tangent spaces of the almost complex manifold is an
isometry (respectively, an anti-isometry). The latter manifolds
are known as \emph{generalized B-manifolds} \cite{GrMeDj} or
\emph{almost complex manifolds with Norden metric} \cite{GaBo} or
\emph{complex Riemannian manifolds} \cite{LeB83}. The Norden
metric is a pseudo-Riemannian metric of neutral signature whereas
the Hermitian metric is Riemannian.

In the odd-dimensional case, the additional direction is spanned
by a vector field $\xi$. Then its dual 1-form $\eta$ determines a
codimension one distribution $H = \ker(\eta)$ endowed with an
almost complex structure $\f$. Then we have an almost contact
structure $(\f,\xi,\eta)$. If the almost complex structure is
equipped with a Hermitian metric then the almost contact
manifold is called \emph{metric}. In the case when the restriction
of the metric on $H$ is a Norden metric
then we deal with an \emph{almost contact manifold with B-metric}
(or an \emph{almost contact complex Riemannian manifold}). Any
B-metric as an odd-dimensional counterpart of a Norden metric is a
pseudo-Riemannian metric of signature $(n+1,n)$.

The goal of the present paper is to survey the research on
canonical-type connections in the case of Norden-type metrics
as well as some additions and generalizations are made. In
Section~\ref{sec:even} we consider the even-dimensional case and
in Section~\ref{sec:odd} --- the odd-dimensional one.

\begin{notation}
\begin{enumerate}
    \item[(a)]
    The notation $\sx$ (or simply $\s$) means
the cyclic sum by the three arguments $x$, $y$, $z$; e.g.,
$\sx F(x,y,z)=F(x,y,z)+F(y,z,x)+F(z,x,y)$;
    \item[(b)]
    For the sake of brevity, we shall use the notation
       $\{A(x,y,z)\}_{[x\leftrightarrow y]}$ for the difference $A(x,y,z)-A(y,x,z)$ and
$\{A(x,y,z)\}_{(x\leftrightarrow y)}$ for the sum
$A(x,y,z)+A(y,x,z)$, where $A$ is an arbitrary tensor of type $(0,3)$;
    \item[(c)] We shall use double subscripts separated by the symbol $\slash$.
    The former and latter subscripts regarding this symbol
correspond to the upper and down signs plus and minus (or, $=$ and $\neq$) in the same
equality, respectively.
For example, the notation %
$\F_{8/9}: F(x,y,z)=F(x,y,\xi)\eta(z)+F(x,z,\xi)\eta(y)$,
$F(x,y,\xi)=\pm F(y,x,\xi)=F(\f x,\f y,\xi)$ means %
$\F_{8}: F(x,y,z)=F(x,y,\xi)\eta(z)+F(x,z,\xi)\eta(y)$,
$F(x,y,\xi)= F(y,x,\xi)=F(\f x,\f y,\xi)$ %
and $\F_{9}: F(x,y,z)=F(x,y,\xi)\eta(z)+F(x,z,\xi)\eta(y)$,
$F(x,y,\xi)=- F(y,x,\xi)=F(\f x,\f y,\xi)$. 
Similarly, $\T_{1/2}: T(\xi,y,z)=T(x,y,\xi)=0,\ T(x,y,z)=-T(\f x,\f y,z)=- T(x,\f y,\f z),
    \ t\nntstile{=}{\neq} 0$ means $\T_{1}: T(\xi,y,z)=T(x,y,\xi)=0,\ T(x,y,z)=-T(\f x,\f y,z)=- T(x,\f y,\f z),
    \ t\neq 0$ and $\T_{2}: T(\xi,y,z)=T(x,y,\xi)=0,\ T(x,y,z)=-T(\f x,\f y,z)=- T(x,\f y,\f z),
    \ t= 0$.
\end{enumerate}
\end{notation}


\section{Almost complex manifolds with Norden metric}\label{sec:even}

Let us consider an almost complex manifold with Norden metric or
an \emph{almost complex Norden manifold} $(M',J,g')$, \ie
\begin{equation}\label{2.1}
J^2x=-x, \qquad g'(Jx,Jy)=-g'(x,y)
\end{equation}
for all differentiable vector fields $x$, $y$ on $M'$. It is
$2n$-dimensional. The \emph{associated metric} $\widetilde{g}'$ of
$g'$ on $M'$ defined by $\widetilde{g}'(x,y)=g'(x,Jy)$ is also a
Norden metric. The signature of both the metrics is necessarily
$(n,n)$.

These manifolds are known as almost complex
manifolds with Norden metric \cite{Nor60,Nor72,GaGrMi85}, almost
complex manifolds with B-metric \cite{GaGrMi87,GaMi87}
or almost complex manifolds with complex Riemannian metric
\cite{LeB83,Manin88,GaIv92,BoFeFrVo99}.
Their structure group is $GL(n,\C)\cap O(n,n)$.

Further in this section, $x$, $y$, $z$, $w$ will stand for
arbitrary differentiable vector fields on $M'$ (or vectors in the
tangent space  of $M'$ at an arbitrary point of $M'$).
Moreover, let $\{e_i\}$ ($i=1,2,\dots,2n$) be an arbitrary
basis of the tangent space of $M'$ at
any point of $M'$ and $g'^{ij}$ be the corresponding components
of the inverse matrix of $g'$.

The fundamental $(0,3)$-tensor $F'$ on $M'$ is defined by
$
F'(x,y,z)=g'\bigl( \left( \n'_x J \right)y,z\bigr)$, 
where $\n'$ is the Levi-Civita connection of $g'$, and $F'$ has the
following properties: \cite{GrMeDj}
\begin{equation}\label{2.3}
F'(x,y,z)=F'(x,z,y)=F'(x,Jy,Jz).
\end{equation}
The corresponding Lee form $\ta'$ is defined by $\ta'(z)=g'^{ij}F'(e_i,e_j,z)$.
The associated trace with respect to the metric $\widetilde g'$ is defined by
$\widetilde \ta'(z)=\widetilde g'^{ij}F'(e_i,e_j,z)$, which implies
the relation $\widetilde \ta'(z)=\ta'(Jz)$ because of
$\widetilde g'^{ij}F'(e_i,e_j,z)=-g'^{ij}F'(e_i,Je_j,z)=g'^{ij}F'(e_i,e_j,Jz)$.

In  \cite{GaBo}, the considered manifolds are classified into three
basic classes $\W_i$ $(i=1,2,3)$ with respect to $F'$.
All classes are determined  as follows:
\begin{equation}\label{Wi}
\begin{array}{rl}
\W_0:\, &F'(x,y,z)=0;\\[0pt]
\W_1:\, &F'(x,y,z)=\frac{1}{2n}\bigl\{g'(x,y)\ta'(z)+g'(x,Jy)\ta'(Jz)\bigr\}
_{(y\leftrightarrow z)};\\[0pt]
\W_2:\, &\sx F'(x,y,Jz)=0,\quad \ta'=0;\\[0pt]
\W_3:\, &\sx F'(x,y,z)=0;\\[0pt]
\W_1\oplus\W_2:\, &\sx F'(x,y,Jz)=0;\\[0pt]
\W_1\oplus\W_3:\, &\sx F'(x,y,z)=\frac{1}{n}\sx\bigl\{g'(x,y)\ta'(z)
+g'(x,Jy)\ta'(Jz)\bigr\};\\[0pt]
\W_2\oplus\W_3:\, &\ta'=0;\\[0pt]
\W_1\oplus\W_2\oplus\W_3:\, &\text{no conditions}.
\end{array}
\end{equation}
The class $\W_0$ of the
\emph{K\"ahler manifolds with Norden metric} belongs to any
other class.


Let $R'$ be the curvature tensor of $\n'$, \ie $R'=[\n'\ , \n'\ ]
- \n'_{[\ ,\ ]}$ and the corresponding $(0,4)$-tensor is
determined by $R'(x,y,z,w)=g'(R'(x,y)z,w)$. The Ricci tensor
$\rho'$ and the scalar curvature $\tau'$  are defined as usual by
$\rho'(y,z)=g'^{ij}R'(e_i,y,z,e_j)$ and
$\tau'=g'^{ij}\rho'(e_i,e_j)$.

A tensor $L$ of type (0,4) having the pro\-per\-ties
$L(x,y,z,w)=-L(y,x,z,w)=-L(x,y,w,z)$, $\mathop{\s} \limits_{x,y,z} L(x,y,z,w)=0 
$ %
is called a \emph{curvature-like tensor}. Moreover, if the
curvature-like tensor $L$ has the property
$L(x,y,Jz,Jw)=-L(x,y,z,w)$,
it is called a \emph{K\"ahler tensor} \cite{GaGrMi87}.

\subsection{The pair of the Nijenhuis tensors}

As it is well known, the Nijenhuis tensor $N'$ of the almost complex structure $J$ is
defined by
\begin{equation}\label{NJ}
N'(x,y) := [J, J](x, y)=\left[Jx,Jy\right]-\left[x,y\right]-J\left[Jx,y\right]-J\left[x,Jy\right].
\end{equation}
Besides it, we define the following symmetric
(1,2)-tensor $\widehat N'$  in analogy to \eqref{NJ} by
\[
\widehat N'(x,y)=\{J ,J\}(x,y)=\{Jx,Jy\}-\{x,y\}-J\{Jx,y\}-J\{x,Jy\},
\]
where the symmetric braces $\{x,y\}=\nabla_xy+\nabla_yx$
are used instead of the antisymmetric brackets $[x,y]=\nabla_xy-\nabla_yx$.
The tensor $\widehat N'$ we also call the  \emph{associated
Nijenhuis tensor} of the almost complex structure.

The pair of the Nijenhuis tensors $N'$ and
$\widehat N'$ plays a fundamental role in the topic
of natural connections (\ie $J$ and $g'$ are parallel with respect to them) on an almost complex Norden manifold.
The torsions and the potentials of these connections
are expressed by these two tensors. By this reason
we characterize the classes of the considered manifolds in terms of $N'$ and
$\widehat N'$.

As it is known from \cite{GaBo}, the class $\W_3$ of
the \emph{quasi-K\"ahler manifolds with Norden
metric} is the only basic class of the considered manifolds with
non-integrable almost complex structure $J$, because $N'$ is
non-zero there. Moreover, this class is determined by the condition
$\widehat N'=0$. The class $\W_1\oplus\W_2$ of the
\emph{(integrable almost) complex manifolds with Norden metric}
is characterized by $N'=0$ and $\widehat N'\neq 0$.
Additionally, the basic classes $\W_1$ and $\W_2$
are distinguish from each other according to the
Lee form $\ta'$: for $\W_1$ the tensor $F'$ is expressed explicitly by the metric and the Lee form, \ie
$\ta'\neq 0$; whereas for $\W_2$ it is valid $\ta'=0$.

The corresponding $(0,3)$-tensors are denoted by the same
letter, $N'(x,y,z)=g'(N'(x,y),z)$, $\widehat N'(x,y,z)=g'(\widehat N'(x,y),z)$.
Both tensors $N'$ and $\widehat N'$ can be expressed in
terms of $F'$ as follows: \cite{GaBo}
\begin{gather}
N'(x,y,z)=F'(x,Jy,z)-F'(y,Jx,z)+F'(Jx,y,z)-F'(Jy,x,z), \label{NF}\\[0pt]%
\widehat N'(x,y,z)=F'(x,Jy,z)+F'(y,Jx,z)+F'(Jx,y,z)+F'(Jy,x,z).\label{NhatF}
\end{gather}
The tensor $\widehat N'$ coincides with the tensor $\widetilde N'$
introduced in \cite{GaBo} by an equivalent equality of
\eqref{NhatF}.

By virtue of \eqref{2.1}, \eqref{2.3}, \eqref{NF} and \eqref{NhatF}, we get the following properties
of  $N'$ and $\widehat N'$:
\begin{gather}
N'(x,y,z)=N'(x,Jy,Jz)=N'(Jx,y,Jz)=-N'(Jx,Jy,z),\\  N'(Jx,y,z)=N'(x,Jy,z)=-N'(x,y,Jz);\label{N-prop}\\
\widehat N'(x,y,z)
=\widehat N'(x,Jy,Jz)=\widehat N'(Jx,y,Jz)
=-\widehat N'(Jx,Jy,z), \\
\widehat N'(Jx,y,z)=\widehat N'(x,Jy,z)=-\widehat N'(x,y,Jz).\label{hatN-prop}
\end{gather}

\begin{theorem}\label{thm:FN}
The fundamental tensor $F'$ of an almost complex Norden manifold $(M',J,g')$
is expressed in terms of the Nijenhuis tensors $N'$ and $\widehat N'$ by the formula
\begin{equation}\label{F=NhatN}
F'(x,y,z)=-\frac14\bigl\{N'(Jx,y,z)+N'(Jx,z,y)
+\widehat N'(Jx,y,z)+\widehat N'(Jx,z,y)\bigr\}.
\end{equation}
\end{theorem}
\begin{proof}
Taking the sum of \eqref{NF} and \eqref{NhatF}, we obtain
\begin{equation}\label{s1}
F'(Jx,y,z)+F'(x,Jy,z)=\frac12\bigl\{N'(x,y,z)+\widehat
N'(x,y,z)\bigr\}.
\end{equation}
The  identities  \eqref{2.1} and \eqref{2.3} imply
\begin{equation}\label{s2}
F'(x,z,Jy)=-F'(x,y,Jz).
\end{equation}
A suitable combination of \eqref{s1} and \eqref{s2} yields
\begin{equation}\label{FJN}
F'(Jx,y,z)=\frac14\bigl\{N'(x,y,z)+N'(x,z,y)+\widehat
N'(x,y,z)+\widehat N'(x,z,y)\bigr\}.
\end{equation}
Applying \eqref{2.1} to \eqref{FJN}, we obtain the stated formula.
\end{proof}

As direct corollaries of \thmref{thm:FN} we have:
\begin{equation}\label{Wi:N0Nhat0}
\begin{array}{rl}
\W_1\oplus\W_2: \; &F'(x,y,z)=-\frac{1}{4}\bigl\{\widehat
N'(Jx,y,z)+\widehat N'(Jx,z,y)\bigr\},\\
\W_{3}: \; &F'(x,y,z)=-\frac{1}{4}\bigl\{N'(Jx,y,z)+N'(Jx,z,y)\bigr\}.
\end{array}
\end{equation}

According to \thmref{thm:FN}, we obtain the following relation for the corresponding traces:
\begin{equation}\label{ta=nu}
\ta'=\frac14\widehat\nu'\circ J,
\end{equation}
where ${\widehat\nu'}(z)={g'}^{ij}\widehat{N}'(e_i,e_j,z)$.
For the traces with respect to the associated metric
$\widetilde{g'}$  of $F'$ and $\widehat N'$, \ie
$\widetilde{\ta'}(z)=\widetilde{g'}^{ij}F'(e_i,e_j,z)$ and
$\widetilde{\widehat\nu'}(z)=\widetilde{g'}^{ij}\widehat N'(e_i,e_j,z)$,
 we have $\widetilde{\ta'}=-\frac14\widehat\nu=\ta'\circ J$
and $\widetilde{\widehat\nu'}=4\ta'=\widehat\nu'\circ J$, respectively.

Then, bearing in mind \eqref{Wi} and the subsequent comments on the pair
of the Nijenhuis tensors, from \thmref{thm:FN} and \eqref{ta=nu} we obtain
immediately the following
\begin{theorem}\label{prop:Wi:N}
The classes of almost complex Norden manifolds
are characterized by the Nijenhuis tensors $N'$ and $\widehat N'$ as follows:
\begin{equation}
\label{Wi:N}
\begin{array}{rll}
\W_{0}: \quad &N'=0, \quad  &\widehat N'=0;\\[0pt]
\W_{1}: \quad &N'=0, \quad
&\widehat N'=\frac{1}{2n}\bigl\{\widehat\nu'\otimes g'+\widetilde{\widehat\nu'}
\otimes\widetilde{g}'\bigr\};\\[0pt]
\W_{2}: \quad &N'=0, \quad
&\widehat\nu'=0;\\[0pt]
\W_{3}: \quad & &\widehat N'=0;\\[0pt]
\W_{1}\oplus\W_2: \quad &N'=0; \quad &
\\[0pt]
\W_{1}\oplus\W_3: \quad &
&\widehat N'=\frac{1}{2n}\bigl\{\widehat\nu'\otimes g'+\widetilde{\widehat\nu'}
\otimes\widetilde{g}'\bigr\};\\[0pt]
\phantom{\W_{1}\oplus}
\W_{2}\oplus\W_3: \quad &&
\widehat\nu'=0;\\
\W_{1}\oplus\W_{2}\oplus\W_3: \quad
&\multicolumn{2}{l}{\quad \text{no
conditions}.}
\end{array}
\end{equation}
\end{theorem}

\subsection{Natural connections on an almost complex Norden manifold}

Let $\nn$ be a linear connection with a torsion $T^*$ and a
potential $Q^*$ with respect to  $\n'$,
\ie
\[
T^*(x,y)=\nn_x y-\nn_y x-[x,y],\quad Q^*(x,y)=\nn_x y-\n'_x y.
\]
The corresponding (0,3)-tensors are defined by
\[    
T^*(x,y,z)=g'(T^*(x,y),z),\quad Q^*(x,y,z)=g'(Q^*(x,y),z).
\]
These tensors have the same mutual relations as in \eqref{TQ} and
\eqref{QT}.

In  \cite{GaMi87}, it is given a partial decomposition of the space
$\mathcal{T}$ of all torsion (0,3)-tensors $T$ (\ie satisfying
$T(x,y,z)=-T(y,x,z)$) on an almost complex Norden manifold
$(M',J,g')$: $\mathcal{T}=\mathcal{T}_1\oplus\mathcal{T}_2
\oplus\mathcal{T}_3\oplus\mathcal{T}_4$, where $\mathcal{T}_i$
$(i=1,2,3,4)$ are invariant orthogonal subspaces with respect to
the structure group $GL(n,\C)\cap O(n,n)$:
\begin{equation*}
  \begin{array}{l}
    \T_1:\quad T(x,y,z)=-T(Jx,Jy,z)=-T(Jx,y,Jz);\\[0pt]
    \T_2:\quad T(x,y,z)=-T(Jx,Jy,z)=T(Jx,y,Jz);\\[0pt]
    \T_3:\quad T(x,y,z)=T(Jx,Jy,z),\quad \sx T(x,y,z)=0,\\[0pt]
    \T_4:\quad T(x,y,z)=T(Jx,Jy,z),\quad \sx T(Jx,y,z)=0.
  \end{array}
\end{equation*}
Moreover, in  \cite{GaMi87} there are explicitly given the
components $T_i$ of $T\in\mathcal{T}$ in $\mathcal{T}_i$
$(i=1,2,3,4)$.

A linear connection $\nn$ on an almost complex manifold with
Norden metric $(M',J,g')$ is called a \emph{natural connection} if
$\nn J=\nn g'=0$. These conditions are equivalent to $\nn g'=\nn
\widetilde{g}'=0$. The connection $\nn$ is natural if and only if
the following conditions for its potential $Q^*$ are valid:
\begin{gather}\label{3.5}
    F'(x,y,z) = Q^*(x,y,Jz)-Q^*(x,Jy,z),\quad
    Q^*(x,y,z) = -Q^*(x,z,y).
\end{gather}

In terms of the components $T_i$, a linear connection with
torsion $T$ on $(M',J,g')$ is
natural if and only if
\[
    T_2(x,y,z) = \frac{1}{4}N'(x,y,z),\quad
    T_3(x,y,z) = \frac{1}{8}\bigl\{\widehat N'(z,y,x)-\widehat N'(z,x,y)\bigr\}.
\]
The former condition
is given in  \cite{GaMi87} whereas the latter one
follows immediately by \eqref{TQ}, \eqref{QT},
\eqref{NhatF} and \eqref{3.5}.

\subsection{The B-connection and the canonical connection}

In \cite{GaGrMi87}, it is introduced the \emph{B-connection} $\dot{\n}'$ only for the manifolds from the class $\W_1$ by
\begin{equation}\label{Tb}
\dot{\n}'_xy=\n'_xy-\frac{1}{2}J\left(\n'_xJ\right)y.
\end{equation}
Obviously, the B-connection is a natural connection on $(M',J,g')$
and it exists in any class of the considered manifolds. Only on a $\W_0$-manifold,
the B-connection coincides with the Levi-Civita connection.

By virtue of \eqref{TQ}, \eqref{N-prop}, \eqref{hatN-prop}, \eqref{F=NhatN},
from \eqref{Tb} we express the torsion of the B-connection as follows:
\begin{equation}\label{TbNhatN}
\dot{T}'(x,y,z)=\frac18 \bigl\{N'(x,y,z)+\sx N'(x,y,z)+\widehat N'(z,y,x)
-\widehat N'(z,x,y)\bigr\}.
\end{equation}

A natural connection with torsion $\ddot{T}'$ on an almost complex
manifold with Norden metric $(M',J,g')$ is called a
\emph{canonical connection} if $\ddot{T}'$ satisfies the following
condition \cite{GaMi87}
\begin{equation}\label{4.1}
    \ddot{T}'(x,y,z)+\ddot{T}'(y,z,x)-\ddot{T}'(Jx,y,Jz)-\ddot{T}'(y,Jz,Jx)=0.
\end{equation}

In  \cite{GaMi87} it is shown that \eqref{4.1} is equivalent to the
condition
$\ddot{T}'_1=\ddot{T}'_4=0$,
\ie $\ddot{T}'\in\mathcal{T}_2\oplus\mathcal{T}_3$. Moreover,
there it is proved that on every almost complex Norden manifold
there exists a unique canonical
connection $\ddot{\n}'$. We express its torsion in terms of $N'$
and $\widehat N'$ as follows
\begin{equation}\label{4.3}
    \ddot{T}'(x,y,z)=\frac{1}{4}N'(x,y,z)
    +\frac{1}{8}\bigl\{\widehat N'(z,y,x)-\widehat N'(z,x,y)\bigr\}.
\end{equation}

Taking into account \eqref{4.3} and \eqref{TbNhatN}, it is easy to
conclude that $\ddot{\n}'\equiv \dot{\n}'$ is valid
if and only if the condition $N'=\s N'$ holds which is equivalent
to $N'=0$. In other words, on a complex Norden manifold, \ie
$(M',J,g')\in\W_1\oplus\W_2$, the
canonical con\-nect\-ion and the B-connection coincide.

Now, let $(M',J,g')$ be in the class $\W_1$. This is the class of the conformally equivalent manifolds of the K\"ahler manifold with Norden
metric. The conformal equivalence is made with respect to the general conformal transformations of the metric $g'$
defined by
\begin{equation}\label{transf}
    \overline g'= e^{2u}\left\{\cos{2v}\ g'+\sin{2v}\ \widetilde g'\right),
\end{equation}
where $u$ and $v$ are differentiable functions on
$M'$~ \cite{GaGrMi87}. For $v=0$ they are restricted to the usual
conformal transformations. The manifold $(M',J,\overline{g}')$ is
again an almost complex Norden  manifold. An important
subgroup of the general group $C$ of the conformal transformations
\eqref{transf} is the group $C_0$ of the \emph{holomorphic
conformal transformations}, defined by the condition: $u + iv$ is
a holomorphic function, \ie $\D u = \D v \circ J$. Then torsion of
the canonical connection is an invariant of $C_0$, \ie the relation
$\overline{\ddot{T}'}(x,y)=\ddot{T}'(x,y)$ holds with respect to any
transformation of $C_0$.
There are
proved that the curvature tensor of the canonical connection is a
K\"ahler tensor if and only if $(M',J,g')\in\W_1^0$, \ie a manifold in $\W_1$
with closed forms $\ta'$ and $\ta'\circ J$.
Moreover, there are studied conformal invariants of the canonical connection in $\W_1^0$.

%
%




Bearing in mind the conformal invariance of both the basic classes and
the torsion $\ddot{T}'$ of the canonical
connection,
the conditions for $\ddot{T}'$ are used in  \cite{GaMi87} for other characteristics of all classes
of the almost complex Norden manifolds as follows:
\begin{equation}\label{Wi:Tc}
\begin{array}{rl}
\W_0:\; &\ddot{T}'(x,y)=0;\\
\W_1:\; &\ddot{T}'(x,y)=\frac{1}{2n}\left\{\ddot{t}'(x)y-\ddot{t}'(y)x
+\ddot{t}'(Jx)Jy-\ddot{t}'(Jy)Jx\right\};\\[0pt]
\W_2:\; &\ddot{T}'(x,y)=\ddot{T}'(Jx,Jy),\quad \ddot{t}'=0;\\[0pt]
\W_3:\; &\ddot{T}'(Jx,y)=-J\ddot{T}'(x,y);\\
\W_1\oplus\W_2:\; &\ddot{T}'(x,y)=\ddot{T}'(Jx,Jy),\quad \sx \ddot{T}'(x,y,z)=0;\\
\W_1\oplus\W_3:\; &\ddot{T}'(Jx,y)+J\ddot{T}'(x,y)=\frac{1}{n}\bigl\{
\ddot{t}'(Jy)x-\ddot{t}'(y)Jx\bigr\};
\\
\W_2\oplus\W_3:\;  &\ddot{t}'=0;\\
\W_1\oplus\W_2\oplus\W_3:\;  &\text{no conditions},
\end{array}
\end{equation}
where $\ddot{t}'(x)=g'^{ij}\ddot{T}'(x,e_i,e_j)$.
The special class $\W_0$ is characterized by the condition $\ddot{T}'(x,y)=0$, \ie
there $\ddot{\n}'\equiv \n'$ holds.

The classes of the almost complex Norden manifolds are determined  with respect to the Nijenhuis tensors in
\eqref{Wi:N}, the same classes
are characterized by conditions for
the torsion of the canonical
connection in \eqref{Wi:Tc}. By virtue of these results we obtain the following
\begin{theorem}\label{thm:Wi:NhatN}
The classes of the almost complex Norden manifolds $M=\M$ are characterized
by an expression of the torsion $\ddot{T}'$ of the canonical
connection in terms of the Nijenhuis tensors $N$ and $\widehat N$ as follows:
\begin{equation}\label{Wi:Tc:NhatN}
\begin{array}{rl}
\W_1:\; &\ddot{T}'(x,y,z)=\frac{1}{16n}\bigl\{
    \widehat \nu'(x)g'(y,z)+\widehat \nu'(Jx)g'(y,Jz)
    \bigr\}_{[x\leftrightarrow y]};\\[0pt]
\W_2:\; &\ddot{T}'(x,y,z)=\frac{1}{8}\bigl\{
    \widehat N'(z,y,x)-\widehat N'(z,x,y)\bigr\},\quad \ddot{t}'=\widehat\nu'=0;\\[0pt]
\W_3:\; &\ddot{T}'(x,y,z)=\frac{1}{4}N'(x,y,z);\\
\W_1\oplus\W_2:\; &\ddot{T}'(x,y,z)=\frac{1}{8}\bigl\{
    \widehat N'(z,y,x)-\widehat N'(z,x,y)\bigr\};\\
\W_1\oplus\W_3:\; &\ddot{T}'(x,y,z)=\frac14N'(x,y,z)\\
&\phantom{\ddot{T}'(x,y,z)}
+\frac{1}{16n}\bigl\{
    \widehat \nu'(x)g'(y,z)+\widehat \nu'(Jx)g'(y,Jz)
    \bigr\}_{[x\leftrightarrow y]};
\\
\W_2\oplus\W_3:\;  &\ddot{T}'(x,y,z)=\frac{1}{4}N'(x,y,z)
    +\frac{1}{8}\bigl\{\widehat N'(z,y,x)-\widehat N'(z,x,y)\bigr\},\\
&    \ddot{t}'=\widehat \nu'=0.
\end{array}
\end{equation}
The special class $\W_0$ is characterized by $\ddot{T}'=0$ and
the whole class $\W_1\oplus\W_2\oplus\W_3$ --- by \eqref{4.3} only.

Moreover, bearing in mind the classifications with respect to the
tensor $F'$ and the torsion $\ddot{T}'$ in \cite{GaBo} and \cite{GaMi87},
respectively, we have:
\begin{itemize}
\item $M\in\W_1\oplus\W_2$ if and only if $\ddot{T}'\in\T_3$;
\item $M\in\W_1$ if and only if $\ddot{T}'\in\T_3^1$, where $\T_3^1$
is the subclass of $\T_3$ with the vectorial torsions\footnote{A \emph{vectorial torsion}
is a torsion which is essentially defined by some vector field on the manifold and its metrics.};
\item $M\in\W_2$ if and only if   $\ddot{T}'\in\T_3^0$,
where $\T_3^0$ is the subclass of $\T_3$ with $\ddot{t}'=0$;
\item $M\in\W_3$ if and only if $\ddot{T}'\in\T_2$;
\item $M\in\W_1\oplus\W_3$ if and only if $\ddot{T}'\in\T_2\oplus\T_3^1$;
\item $M\in\W_2\oplus\W_3$ if and only if $\ddot{T}'\in\T_2\oplus\T_3^0$;
\item $M\in\W_1\oplus\W_2\oplus\W_3$ if and only if $\ddot{T}'\in\T_2\oplus\T_3$.
\end{itemize}
\end{theorem}
\begin{proof}
Let $(M',J,g')$ be a complex Norden manifold, \ie $(M',J,g')\in\W_1\oplus\W_2$.
According to \eqref{4.3} and $N'=0$ in this case, we have
$\ddot{T}'=\ddot{T}'_3$, \ie
    $\ddot{T}'\in\mathcal{T}_3$ and the expression  $\ddot{T}'(x,y,z)=\frac{1}{8}\bigl\{
    \widehat N'(z,y,x)-\widehat N'(z,x,y)\bigr\}$ is obtained.
Applying \eqref{Wi:N} to the latter equality, we determine the basic classes
$\W_1$ and $\W_2$ as is given in \eqref{Wi:Tc:NhatN} and the corresponding
subclasses $\T_3^1$ and $\T_3^0$, respectively.
Taking into account the relation between the corresponding traces
$\widehat \nu'=8\ddot{t}'$, which is a consequence of the equality for
$\W_1\oplus\W_2$, we obtain the  characterization for these two basic classes in \eqref{Wi:Tc}.

Let $(M',J,g')$ be a quasi-K\"ahler manifold with Norden metric, \ie $(M',J,g')\in\W_3$.
By virtue of \eqref{4.3} and $\widehat N'=0$ for such a manifold,
we have $\ddot{T}'=\ddot{T}'_2$, \ie
    $\ddot{T}'\in\mathcal{T}_2$ and therefore we give
$\ddot{T}'=\frac{1}{4}N'$.
Obviously, the form of $\ddot{T}'$ in the latter equality satisfies
the condition for $\W_3$ in \eqref{Wi:Tc}.

In a similar way we get for the rest classes $\W_1\oplus\W_3$ and $\W_2\oplus\W_3$.
The conditions of these two classes, given in \eqref{Wi:Tc},
are consequences of the corresponding equalities in \eqref{Wi:Tc:NhatN}.
The case of the whole class $\W_1\oplus\W_2\oplus\W_3$ was discussed above.
\end{proof}

The canonical connections on quasi-K\"ahler manifolds with Norden
metric are considered in more details in  \cite{Mek09}. There are given the following formulae for
the potential $\ddot{Q}'$ and the torsion $\ddot{T}'$ on a
$\W_3$-manifold:
\[
\begin{array}{l}
    \ddot{Q}'(x,y) = \frac{1}{4}\left\{\left(\n'_y J\right)Jx-\left(\n'_{Jy} J\right)x
    +2\left(\n'_{x}
    J\right)Jy\right\},\\
    \ddot{T}'(x,y) = \frac{1}{2}\left\{\left(\n'_x J\right)Jy+\left(\n'_{Jx} J\right)y\right\}.
\end{array}
\]
%
Moreover, some properties for the curvature and the torsion of the  canonical connection are obtained.




\subsection{The KT-connection}

In  \cite{Mek-2}, it is proved that a natural connection $\dddot{\n'}$
with totally skew-symmetric torsion, called a
\emph{KT-connection}, exists on an almost complex Norden manifold
$(M',J,g')$ if and only if $(M',J,g')$ belongs to $\W_3$, \ie the
manifold is quasi-K\"ahlerian with Norden metric. Moreover,
the KT-connection is unique and it is determined by its potential
\begin{equation}\label{QQQ3}
    \dddot{Q'}(x,y,z)=-\frac{1}{4}\mathop{\s} \limits_{x,y,z} F'(x,y,Jz).
\end{equation}

As mentioned above, the
canonical con\-nect\-ion and the B-connection coincide on $(M',J,g')\in\W_1\oplus\W_2$
whereas the
KT-connection does not exist there.

The following natural connections on $(M',J,g')$ are studied on a quasi-K\"ahler
manifold with Norden metric: the B-connection $\dot{\n}'$
(\cite{Mek-1}), the
canonical connection $\ddot{\n}'$ (\cite{Mek09}) and the KT-connection $\dddot{\n}'$  (\cite{Mek-2}).

From the relations \eqref{4.3} and \eqref{TbNhatN} for a $\W_3$-manifold follow
\begin{equation}\label{TcTbNW3}
\dot{T}'(x,y,z)=\frac{1}{8}\bigl\{N'(x,y,z)+\sx N'(x,y,z)\bigr\},\quad
\ddot{T}'(x,y,z)=\frac{1}{4}N'(x,y,z).
\end{equation}
The equalities \eqref{TQ} and \eqref{QQQ3} yield $\dddot{T}'(x,y,z)=
-\frac{1}{2}\mathop{\s} \limits_{x,y,z} F'(x,y,Jz)$, which by
\eqref{Wi:N0Nhat0} for $\W_3$ and \eqref{N-prop} implies
\begin{equation}\label{TkNW3}
\dddot{T}'(x,y,z)=\frac{1}{4}\sx N'(x,y,z).
\end{equation}

Then from \eqref{TcTbNW3} and \eqref{TkNW3} we have  the relation
$\dot{T}'=\frac{1}{2}\left(\ddot{T}'+ \dddot{T}'\right)$, which by
\eqref{QT} is equivalent to $\dot{Q}'=\frac{1}{2}\left(\ddot{Q}'+ \dddot{Q}'\right)$.
Therefore, as it is shown in  \cite{Mek09}, the B-connection
is the \emph{average} connection for the canonical connection and the
KT-connection on a quasi-K\"ahler manifold with Norden metric
, \ie
$
\dot{\n'}=\frac{1}{2}\left(\ddot{\n'}+ \dddot{\n'}\right)$.


\section{Almost contact manifolds with B-metric}\label{sec:odd}

Let $(M,\f,\xi,\eta)$ be an almost contact manifold,  \ie  $M$ is
a $(2n+1)$-dimen\-sion\-al differentiable manifold with an almost
contact structure $(\f,\xi,\eta)$ consisting of an endomorphism
$\f$ of the tangent bundle, a vector field $\xi$ and its dual
1-form $\eta$ such that the following relations are valid:
\begin{equation}\label{str1}
\f\xi = 0,\quad \f^2 = -\Id + \eta \otimes \xi,\quad
\eta\circ\f=0,\quad \eta(\xi)=1.
\end{equation}
Later on, let us equip $(M,\f,\xi,\eta)$ with a pseudo-Riemannian
metric $g$ of signature $(n+1,n)$ determined by
\begin{equation}\label{str2}
g(\f x, \f y ) = - g(x, y ) + \eta(x)\eta(y)
\end{equation}
for arbitrary differentiable vector fields $x$, $y$ on $M$.
Then $(M,\f,\xi,\eta,g)$ is called an almost contact manifold with
B-metric or an \emph{almost contact B-metric manifold}.
The associated metric $\widetilde{g}$ of $g$ on $M$ is defined by the equality
$\widetilde{g}(x,y)=g(x,\f y)\allowbreak+\eta(x)\eta(y)$. Both
metrics $g$ and $\widetilde{g}$ are necessarily of signature
$(n+1,n)$. The manifold $(M,\f,\xi,\eta,\widetilde{g})$ is also an
almost contact B-metric manifold. \cite{GaMiGr}

Let us remark that the $2n$-dimensional contact distribution
$H=\ker(\eta)$, generated by the contact 1-form $\eta$, can be
considered as the horizontal distribution of the sub-Riemannian
manifold $M$. Then $H$ is endowed with an almost complex structure
determined as $\f|_H$ -- the restriction of $\f$ on $H$, as well
as a Norden metric $g|_H$, \ie
$g|_H(\f|_H\cdot,\f|_H\cdot)=-g|_H(\cdot,\cdot)$. Moreover, $H$
can be considered as an $n$-dimensional complex Riemannian
manifold with a complex Riemannian metric
$g^{\C}=g|_H+i\widetilde{g}|_H$~ \cite{GaIv92}.
By this reason we refer to these manifolds as \emph{almost contact complex Riemannian manifolds}.
They are investigated and studied
in \cite{GaMiGr,Man4,Man31,ManGri1,ManGri2,ManIv38, ManIv36,NakGri2}.
The structure group of these manifolds is
$\left(GL(n,\mathbb{C})\cap
O(n,n)\right)\allowbreak\times I_1$.

Further in this section, $x$, $y$, $z$  will stand for
arbitrary differentiable vector fields on $M$ (or vectors in the
tangent space of $M$ at an arbitrary point of $M$). Moreover,
let $\left\{e_i;\xi\right\}_{i=1}^{2n}$ denote an arbitrary basis  of the
tangent space of $M$ at an arbitrary point in $M$ and $g^{ij}$
be the corresponding components of the inverse matrix of $g$.


The fundamental tensor $F$ of type (0,3) on the manifold $\M$ is defined by
$
F(x,y,z)=g\bigl(\left( \nabla_x \f \right)y,z\bigr)$,
where $\n$ is the Levi-Civita connection for $g$ and the following properties are valid: \cite{GaMiGr}
\begin{equation}\label{F-prop}
F(x,y,z)=F(x,z,y)=F(x,\f y,\f z)+\eta(y)F(x,\xi,z) +\eta(z)F(x,y,\xi).
\end{equation}
The relations of the covariant derivatives $\n\xi$ and $\n\eta$ with $F$ are:
\[
    \left(\n_x\eta\right)y=g\left(\n_x\xi,y\right)=F(x,\f y,\xi).
\]

The following 1-forms, called Lee forms, are associated with $F$:
\begin{equation*}
\begin{array}{c}
\ta(z)=g^{ij}F(e_i,e_j,z),\quad \ta^*(z)=g^{ij}F(e_i,\f
e_j,z),\quad \om(z)=F(\xi,\xi,z).
\end{array}
\end{equation*}
Obviously, the equalities $\ta^*\circ\f=-\ta\circ\f^2$ and
$\om(\xi)=0$ are valid. For the corresponding traces $\widetilde\ta$
and $\widetilde\ta^*$ with respect to $\widetilde g$ we have
$\widetilde\ta=-\ta^*$ and $\widetilde\ta^*=\ta$.

A classification with respect to $F$ of the almost contact
B-metric manifolds is given in  \cite{GaMiGr}. This classification
includes eleven basic classes $\F_1$, $\F_2$, $\dots$, $\F_{11}$.
Their intersection is the special class $\F_0: F(x,y,z)=0$.

Further, we use the following characteristic conditions of the
basic classes: \cite{GaMiGr,Man8}
\begin{equation*}
\begin{array}{rl}
\F_{1}: &F(x,y,z)=\frac{1}{2n}\bigl\{g(x,\f y)\ta(\f z)+g(\f
x,\f y)\ta(\f^2 z)
\bigr\}_{(y\leftrightarrow z)};\\[0pt]
\F_{2}: &F(\xi,y,z)=F(x,\xi,z)=0,\quad
              \sx F(x,y,\f z)=0,\quad \ta=0;\\[0pt]
\F_{3}: &F(\xi,y,z)=F(x,\xi,z)=0,\quad
              \sx F(x,y,z)=0;\\[0pt]
\F_{4}: &F(x,y,z)=-\frac{\ta(\xi)}{2n}\bigl\{g(\f x,\f y)\eta(z)
+g(\f x,\f z)\eta(y)\bigr\};\\[0pt]
\F_{5}: &F(x,y,z)=-\frac{\ta^*(\xi)}{2n}\bigl\{g( x,\f y)\eta(z)
+g(x,\f z)\eta(y)\bigr\};\\[0pt]
\F_{6/7}: &F(x,y,z)=F(x,y,\xi)\eta(z)+F(x,z,\xi)\eta(y),\quad \\[0pt]
                &F(x,y,\xi)=\pm F(y,x,\xi)=-F(\f x,\f y,\xi),\quad \ta=\ta^*=0; \\[0pt]
\F_{8/9}: &F(x,y,z)=F(x,y,\xi)\eta(z)+F(x,z,\xi)\eta(y),\quad
\\[0pt]
                &F(x,y,\xi)=\pm F(y,x,\xi)=F(\f x,\f y,\xi); \\[0pt]
\F_{10}: &F(x,y,z)=F(\xi,\f y,\f z)\eta(x); \\[0pt]
\F_{11}:
&F(x,y,z)=\eta(x)\left\{\eta(y)\om(z)+\eta(z)\om(y)\right\}.
\end{array}
\end{equation*}

\subsection{The pair of the Nijenhuis tensors}

An almost contact structure $(\f,\xi,\eta)$ on $M$ is called
\emph{normal} and respectively $(M,\f,\xi,\eta)$ is a \emph{normal
almost contact manifold} if the corresponding almost complex
structure $J'$ generated on $M'=M\times \R$ is integrable 
 \cite{SaHa}. The almost contact structure is normal if and only if
the Nijenhuis tensor of $(\f,\xi,\eta)$ is zero \cite{Blair}.

The Nijenhuis tensor $N$ of the almost contact structure is
defined by
$N := [\f, \f]+ \D{\eta}\otimes\xi$,
where $[\f, \f](x, y)=\left[\f x,\f
y\right]+\f^2\left[x,y\right]-\f\left[\f x,y\right]-\f\left[x,\f
y\right]$ and $\D{\eta}$ is the exterior derivative of $\eta$.

In \cite{ManIv36}, it is defined the symmetric (1,2)-tensor $\widehat N$
for a $(\f,\xi,\eta)$-structure by
$\widehat N=\{\f,\f\}+(\mathcal L_{\xi}g)\otimes \xi$,
where $\mathcal L$ denotes the Lie derivative and $\{\f ,\f\}$ is given by
$\{\f ,\f\}(x,y)=\{\f x,\f y\}+\f^2\{x,y\}-\f\{\f
x,y\}-\f\{x,\f y\}$ for $\{x,y\}=\nabla_xy+\nabla_yx$.
The tensor $\widehat N$ is also called the  \emph{associated
Nijenhuis tensor} for $(\f,\xi,\eta)$.

Obviously, $N$ is antisymmetric and $\widehat N$ is symmetric, \ie
$
N(x,y)=-N(y,x)$ and $\widehat N(x,y)=\widehat N(y,x)$.

The Nijenhuis tensors $N$ and
$\widehat N$ play a fundamental role in natural connections (\ie such connections that the tensors of the structure $(\f,\xi,\eta,g)$  are parallel with respect to them) on an almost contact B-metric manifold.
The torsions and the potentials of these connections
are expressed by these two tensors. By this reason
we characterize the classes of the considered manifolds in terms of $N$ and
$\widehat N$.

The corresponding tensors of type (0,3) are denoted by the same
letters as follows 
$N(x,y,z)=g(N(x,y),z)$, $\widehat N(x,y,z)=g(\widehat N(x,y),z)$.
Both tensors $N$ and $\widehat N$ are expressed in
terms of $F$ as follows  \cite{ManIv36}
\begin{gather}
N(x,y,z) = \bigl\{F(\f x,y,z)-F(x,y,\f z)+\eta(z)F(x,\f
y,\xi)\bigr\}_{[x\leftrightarrow
y]},\label{enu}\\[0pt]
\widehat N(x,y,z) = \bigl\{F(\f x,y,z)-F(x,y,\f z)+\eta(z)F(x,\f
y,\xi)\bigr\}_{(x\leftrightarrow y)}.\label{enhat}
\end{gather}

Bearing in mind \eqref{str1}, \eqref{str2} and \eqref{F-prop}, from \eqref{enu} and \eqref{enhat}
 we obtain the following properties of the Nijenhuis tensors on
 an arbitrary almost contact B-metric manifold:
\begin{gather}
N(x, \f y,\f z)=N(x, \f^2 y,\f^2 z),\quad N(\f x, y,\f z)=
N(\f^2 x, y,\f^2 z),\nonumber\\ 
N(\f x,\f  y, z)=-N(\f^2 x,\f^2 y,z),\quad
\widehat N(x, \f y,\f z)=\widehat N(x, \f^2 y,\f^2 z),\nonumber\\
\widehat N(\f x, y,\f z)=\widehat N(\f^2 x, y,\f^2 z),\quad
\widehat N(\f x,\f  y, z)=-\widehat N(\f^2 x,\f^2 y,z),\nonumber\\
{N(\xi, \f y,\f z)+N(\xi, \f z,\f y)+\widehat N(\xi, \f y,\f z)
+\widehat N(\xi, \f z,\f y)=0.}\nonumber
\end{gather}

It is known that the class of the normal almost contact B-metric
manifolds, \ie $N=0$, is
$\F_1\oplus\F_2\oplus\F_4\oplus\F_5\oplus\F_6$.
According to \cite{ManIv36}, the class of the almost contact
B-metric manifolds with
$\widehat{N}=0$ is $\F_3\oplus\F_7$.
The latter two statements follow from \cite{ManIv38} and \cite{ManIv36},
where the following form of the Nijenhuis tensors
for each of the basic classes $\F_i$ $(i=1,2,\dots,11)$ of $\M$ is given:
\begin{equation}\label{Fi:N}
\begin{array}{rl}
\F_1:\quad &N(x,y)=0,\quad \widehat N(x,y)=\frac{2}{n}\bigl\{g(\f x,\f y)\f \ta^{\sharp}+g(x,\f
y)\ta^{\sharp}\bigr\};\\
\F_2:\quad &N(x,y)=0,\quad \widehat N(x,y)=2\left\{\left(\n_{\f
x}\f\right)y-\f\left(\n_{x}\f\right)y\right\};\\
\F_3:\quad &N(x,y)=2\left\{\left(\n_{\f
x}\f\right)y-\f\left(\n_{x}\f\right)y\right\},\quad
\widehat N(x,y)=0;\\
\F_4:\quad &N(x,y)=0,\quad \widehat N(x,y)=\frac{2}{n}\ta(\xi)g(x,\f y) \xi;\\
\F_5:\quad &N(x,y)=0,\quad \widehat N(x,y)=-\frac{2}{n}\ta^*(\xi)g(\f x,\f y) \xi;\\
\F_6:\quad &N(x,y)=0,\quad \widehat N(x,y)=4\left(\n_{x}\eta\right)y\ \xi;\\
\F_7:\quad &N(x,y)=4\left(\n_{x}\eta\right)y\ \xi,\quad
\widehat N(x,y)=0;\\
\F_8:\quad &N(x,y)=2\left\{\eta(x)\n_{y}\xi-\eta(y)\n_{x}\xi\right\},\quad\\ 
&\widehat N(x,y)=-2\left\{\eta(x) \n_{y}\xi+\eta(y)
\n_{x}\xi\right\};\\
\F_9:\quad &N(x,y)=2\left\{\eta(x)\n_{y}\xi-\eta(y)\n_{x}\xi\right\},\quad \\
&\widehat N(x,y)=-2\left\{\eta(x) \n_{y}\xi+\eta(y)
\n_{x}\xi\right\};\\
\F_{10}:\quad &N(x,y)=-\eta(x)\f\left(\n_{\xi}\f\right)y+\eta(y)\f\left(\n_{\xi}\f\right)x,\quad \\
&\widehat N(x,y)=-\eta(x)\f \left(
\n_{\xi}\f\right)y-\eta(y)\f \left( \n_{\xi}\f\right)x;\\
\F_{11}:\quad &N(x,y)=\left\{\eta(x)\om(\f y)-\eta(y)\om(\f x)\right\}\xi,\quad \\
&\widehat N(x,y)=\left\{\eta(x)\om(\f y)+\eta(y)\om(\f
x)\right\} \xi,
\end{array}
\end{equation}
where $\ta^{\sharp}$
and $\om^{\sharp}$ are the corresponding vectors of $\ta$ and $\om$ with respect to $g$.

In \cite{IvMaMa14}, the tensor $F$ is expressed by the Nijenhuis tensors
on an arbitrary $\M$ as follows:
\begin{equation}\label{FNhatN}
\begin{split}
F(x,y,z)&=-\frac14\bigl\{N(\f x,y,z)+N(\f x,z,y)
+\widehat N(\f x,y,z)+\widehat N(\f x,z,y)\bigr\}\\
&\phantom{=\ }+\frac12\eta(x)\bigl\{N(\xi,y,\f z)+\widehat
N(\xi,y,\f z)+\eta(z)\widehat N(\xi,\xi,\f y)\bigr\}.
\end{split}
\end{equation}
As corollaries, in the cases when $N=0$ or $\widehat N=0$,
the latter relation takes the following form, respectively:
\begin{gather}
F(x,y,z)=-\frac14\bigl\{\widehat N(\f x,y,z)+\widehat N(\f x,z,y)\bigr\}\nonumber\\
\phantom{F(x,y,z)=}
+\frac12\eta(x)\bigl\{\widehat
N(\xi,y,\f z)+\eta(z)\widehat N(\xi,\xi,\f y)\bigr\},\label{F_hatN}\nonumber\\
F(x,y,z)=-\frac14\bigl\{N(\f x,y,z)+N(\f x,z,y)\bigr\}
+\frac12\eta(x)N(\xi,y,\f z).\label{F_N}\nonumber
\end{gather}

\subsection{Natural connections on an almost contact B-metric manifold}

Let $D$ be a linear connection on $\M$ and let us denote its torsion and potential (with respect to $\n$)
by $T$ and $Q$, respectively.
The corresponding tensors of type (0,3) are determined by
$T(x,y,z)=g(T(x,y),z)$ and $Q(x,y,z)=g\left(Q(x,y),z\right)$. The
relations \eqref{TQ} and \eqref{QT} are valid.

In  \cite{ManIv36}, it is given a classification of all linear
connections on the almost contact B-metric manifolds with respect
to their torsions $T$ in 15 basic classes $\T_{i}$ $(i=1,\dots,15)$ (which are invariant and orthogonal subspaces with respect to
the structure group) as
follows:
\begin{equation*}
\begin{split}
\T_{1/2}
:\quad
    &T(\xi,y,z)=T(x,y,\xi)=0,\quad \\&T(x,y,z)=-T(\f x,\f y,z)=- T(x,\f y,\f z),
    \quad t\nntstile{=}{\neq} 0;\\[0pt]
\T_{3}
:\quad
    &T(\xi,y,z)=T(x,y,\xi)=0,\quad \\&T(x,y,z)=-T(\f x,\f y,z)= T(x,\f y,\f z);\\[0pt]
\T_{4/5}
:\quad &T(\xi,y,z)=T(x,y,\xi)=0,\quad \\&T(x,y,z)-T(\f x,\f y,z)=\sx
T(x,y,z)=0,\quad
t\nntstile{=}{\neq} 0;\\[0pt]
\T_{6}
:\quad &T(\xi,y,z)=T(x,y,\xi)=0,\quad \\&T(x,y,z)-T(\f x,\f y,z)=\sx T(\f x,y,z)=0;\\[0pt]
\T_{7/8}
:\quad &T(x,y,z)=\eta(z)T(\f^2 x,\f^2 y,\xi),\quad
                T(x,y,\xi)=\mp T(\f x,\f y,\xi);\\[0pt]
\end{split}
\end{equation*}
\begin{equation*}
\begin{split}
\T_{9/10}
:\quad &T(x,y,z)=\eta(x)T(\xi,\f^2 y,\f^2 z)-\eta(y)T(\xi,\f^2 x,\f^2 z),\quad \\[0pt]
                &T(\xi,y,z)= T(\xi,z,y)=-T(\xi,\f y,\f z),
                \quad t\nntstile{=}{\neq} 0,\quad t^*\nntstile{\neq}{=} 0; \\[0pt]
\T_{11}
:\quad &T(x,y,z)=\eta(x)T(\xi,\f^2 y,\f^2 z)-\eta(y)T(\xi,\f^2 x,\f^2 z),\quad \\[0pt]
                &T(\xi,y,z)= T(\xi,z,y)=-T(\xi,\f y,\f z),\quad t= 0,\quad t^*=0; \\[0pt]
\T_{12}
:\quad &T(x,y,z)=\eta(x)T(\xi,\f^2 y,\f^2 z)-\eta(y)T(\xi,\f^2
x,\f^2 z),\quad\\&
T(\xi,y,z)=- T(\xi,z,y)=-T(\xi,\f y,\f z); \\[0pt]
\T_{13/14}
:\quad &T(x,y,z)=\eta(x)T(\xi,\f^2 y,\f^2 z)-\eta(y)T(\xi,\f^2
x,\f^2 z),
\quad \\& T(\xi,y,z)=\pm T(\xi,z,y)=T(\xi,\f y,\f z); \\[0pt]
\T_{15}
:\quad
&T(x,y,z)=\eta(z)\left\{\eta(y)\hat{t}(x)-\eta(x)\hat{t}(y)\right\},
\end{split}
\end{equation*}
where the torsion forms associated with $T$ are defined by
\begin{equation*}
t(x)=g^{ij}T(x,e_i,e_j),\quad t^{*}(x)=g^{ij}T(x,e_i,\f e_j),\quad
\hat{t}(x)=T(x,\xi,\xi).
\end{equation*}
Moreover, in  \cite{ManIv36} there are explicitly given the
components $T_i$ of $T\in\mathcal{T}$ in $\mathcal{T}_i$
$(i=1,\dots,15)$.

A linear connection $D$ is called a \emph{natural connection} on
$(M,\f,\allowbreak\xi,\eta,g)$ if the almost contact structure and
the B-metric are parallel with respect to
$D$,  \ie  $D\f=D\xi=D\eta=Dg=0$ \cite{Man31}.
As a corollary, we have also $D\widetilde{g}=0$.
According to  \cite{ManIv36}, a necessary and sufficient condition
for  a linear connection $D$ to be natural on $\M$ is $D\f=Dg=0$.

It is easy to establish  (see, e.g.  \cite{Man31}) that
%
a linear connection $D$ is a natural connection on an almost contact
B-metric manifold if and only if %
\begin{equation*}
 Q(x,y,\f z)-Q(x,\f y,z)=F(x,y,z),\quad
 Q(x,y,z)=-Q(x,z,y).
\end{equation*}
%

Let us remark that the condition a linear connection to be natural does not
imply that some of the basic classes $\T_i$ $(i=1,\dots,15)$ to be empty for natural connections.

In \cite{ManIv36}, it is proved that an almost contact B-metric manifold $M=\M\in\F_i\setminus\F_0$ is
normal, \ie $N=0$, (respectively, has
$\widehat{N}=0$) if the torsion of an arbitrary natural
connection on $M$ belongs to
$\T_{4}\oplus\T_{5}\oplus\T_{9}\oplus\T_{10}\oplus\T_{11}$ (respectively, $\T_{3}\oplus\T_{7}$).

\subsection{The $\f$B-connection and the $\f$-canonical connection}

In  \cite{ManGri2}, it is introduced a natural connection on $\M$
by
\begin{equation}\label{fB}
    \dot{D}_xy=\n_xy+\dot{Q}(x,y),\quad \dot{Q}(x,y)=
    \frac{1}{2}\bigl\{\left(\n_x\f\right)\f
y+\left(\n_x\eta\right)y\ \xi\bigr\}-\eta(y)\n_x\xi.
\end{equation}
In
 \cite{ManIv37}, the connection determined by \eqref{fB} is called
a \emph{$\f$B-connection}. It is studied for some classes of the considered manifolds
in  \cite{ManGri1,ManGri2,Man3,Man4,ManIv37} with respect to  properties
of the torsion and the curvature as well as the conformal geometry. The restriction of the
$\f$B-connection $\dot{D}$ on $H$ coincides with the B-connection
$\dot{\n}'$ on the corresponding almost complex Norden manifold,
given in \eqref{Tb} and studied for the class $\W_1$ in  \cite{GaGrMi87}.

The torsion of the $\f$B-connection has the form
\begin{equation}\label{T0}
\dot{T}(x,y,z) = \frac{1}{2}\bigl\{F(x,\f y,z)+\eta(z)F(x,\f y,\xi)+2\eta(x)F(y,\f
z,\xi)\bigr\}_{[x\leftrightarrow y]}.
\end{equation}
Then it belongs to
$\T_{3}\oplus\T_{4}\oplus\cdots\oplus\T_{15}$, according to \cite{ManIv36}.



Using \eqref{FNhatN}, \eqref{T0} and  the 
orthonormal decomposition $x=hx+vx$, where $hx=-\f^2x$, $vx=\eta(x)\xi$, we give the expression of the torsion of the $\f$B-connection in terms of the Nijenhuis tensors as follows
\begin{equation}\label{T0N}
\begin{split}
\dot{T}(x,y,z) &=\frac{1}{8}\bigl\{N(hx,hy,hz)+\sx N(hx,hy,hz)+\N(hz,hy,hx)-\N(hz,hx,hy)\bigr\}\\
&+\frac{1}{4}\bigl\{2N(vx,hy,hz)+ N(hy,hz,vx)+2\N(vx,hy,hz)+ N(hy,hz,vx)\\
&
+N(hx,hy,vz)+ N(vz,hx,hy)-\N(vz,hx,hy)-2\N(vz,vx,hy)\bigr\}_{[x\leftrightarrow y]}.
\end{split}
\end{equation}

Taking into account \eqref{T0}, \eqref{T0N} and \eqref{Fi:N},
we obtain for the manifolds from $\F_3\oplus\F_7$ the following
\begin{equation}\label{T0N-F37}
\begin{split}
\dot{T}(x,y,z) &=\frac{1}{8}\bigl\{N(hx,hy,hz)+\sx N(hx,hy,hz)\bigr\}
\\
&+\frac{1}{4}\bigl\{N(hx,hy,vz)+\sx N(hx,hy,vz)\bigr\}.
\end{split}
\end{equation}
Therefore, using the notation $hN(x,y,z)=N(hx,hy,hz)$,
for the basic classes with vanishing $\N$ we have:
\begin{equation}\label{T0N-F3F7}
\F_3:\quad \dot{T}=\frac{1}{8}\bigl\{hN+\s hN\bigr\},\qquad
\F_7:\quad
\dot{T}=\frac{1}{2}\bigl\{\D\eta\otimes\eta+\eta\wedge\D\eta\bigr\}.
\end{equation}

A natural connection $\ddot{D}$ is called a \emph{$\f$-canonical
connection} on $(M,\f,\xi,\allowbreak\eta,g)$ if its torsion
$\ddot{T}$ satisfies the following identity: \cite{ManIv38}
\begin{equation*}
\begin{split}
    \bigl\{\ddot{T}(x,y,z)&-\ddot{T}(x,\f y,\f z)
    -\eta(x)\left\{\ddot{T}(\xi,y,z)
    -\ddot{T}(\xi, \f y,\f z)\right\}\\[0pt]
    &-\eta(y)\left\{\ddot{T}(x,\xi,z)-\ddot{T}(x,z,\xi)-\eta(x)
    \ddot{T}(z,\xi,\xi)\right\}\bigr\}_{[y\leftrightarrow z]}=0.
\end{split}
\end{equation*}

Let us remark that the restriction of the $\f$-canonical
connection $\ddot{D}$  on the contact distribution $H$ is
the unique canonical connection $\ddot{\n}'$ with torsion given in \eqref{4.3}
on the corresponding almost
complex Norden manifold studied in  \cite{GaMi87}.

In \cite{ManIv38}, it is constructed a linear connection $\ddot{D}$ as
follows:
\begin{equation*}
\begin{split}
&g(\ddot{D}_xy,z)=g(\n_xy,z)+\ddot{Q}(x,y,z),
\quad\\
&\ddot{Q}(x,y,z)
=\dot{Q}(x,y,z)-\frac{1}{8}\left\{N(\f^2 z,\f^2 y,\f^2 x)+2N(\f z,\f
y,\xi)\eta(x)\right\}.
\end{split}
\end{equation*}
It is a natural connection on $\M$ and its torsion is
\begin{equation*}
\ddot{T}(x,y,z)=\dot{T}(x,y,z)
+\frac{1}{8}\left\{
N(hz,hy,hx)+2N(hz,hy,vx)
\right\}_{[x\leftrightarrow y]},
\end{equation*}
which is equivalent to
\begin{equation}\label{T-can-N}
\begin{split}
\ddot{T}(x,y,z)=\dot{T}(x,y,z)
&+\frac{1}{8}\bigl\{N(hx,hy,hz)-\sx N(hx,hy,hz)\bigr\}\\
&+\frac{1}{4}\bigl\{N(hx,hy,vz)-\sx N(hx,hy,vz)\bigr\}.
\end{split}
\end{equation}
Obviously, $\ddot{D}$ is a $\f$-canonical connection on $\M$ and
it is unique.
Moreover, the torsion forms of the $\f$-canonical connection coincide with those of the  $\f$B-connection.

In  \cite{ManIv38}, it is proved that the $\f$-canonical connection
and the $\f$B-connection coincide on an almost con\-tact B-metric
manifold if and only if $N(hx,hy)$ vanishes, or equivalently, $\M$ belongs to the class
$\U_0=
\F_1\oplus\F_2\oplus\F_4\oplus\F_5\oplus\F_6\oplus\F_8\oplus
\F_9\oplus\F_{10}\oplus\F_{11}$.
In other words, bearing in mind \eqref{T0N},
the torsions of the $\f$-canonical connection and the $\f$B-connection on a
manifold from $\U_0$ have the form
\begin{equation*}\label{T0N=}
\begin{split}
\ddot{T}(x,y,z) =\dot{T}(x,y,z) &=\frac{1}{8}\bigl\{\N(hz,hy,hx)-\N(hz,hx,hy)\bigr\}\\
&+\frac{1}{4}\bigl\{2N(vx,hy,hz)+N(vz,hx,hy)\\
&\phantom{=}
+2\N(vx,hy,hz)-\N(vz,hx,hy)-2\N(vz,vx,hy)\bigr\}_{[x\leftrightarrow y]}.
\end{split}
\end{equation*}

The torsions $\dot{T}$ and $\ddot{T}$ are different each other on a
manifold belonging to the only basic classes $\F_3$ and $\F_7$
as well as to their direct sums with other classes.
For $\F_3\oplus\F_7$, using \eqref{T0N-F37} and \eqref{T-can-N},
we obtain the form of the torsion of the $\f$-canonical connection as follows
\begin{equation*}
\ddot{T}(x,y,z)=\frac{1}{4} N(hx,hy,hz)+\frac{1}{2} N(hx,hy,vz).
\end{equation*}
Therefore,
using \eqref{Fi:N}, the torsion of the $\f$-canonical connection for $\F_3$ and $\F_7$ is expressed by
\begin{equation}\label{Tcan-F3F7}
\F_3:\quad \ddot{T}=\frac{1}{4} hN,\qquad \F_7:\quad \ddot{T}=\D\eta\otimes\eta.
\end{equation}

The general contactly conformal transformations of an almost
contact B-metric structure are defined by
\begin{equation}\label{Transf}
\begin{split}
    &\bar{\xi}=e^{-w}\xi,\quad \bar{\eta}=e^{w}\eta,\quad\\
    &\bar{g}(x,y)=e^{2u}\cos{2v}\ g(x,y)+e^{2u}\sin{2v}\ g(x,\f y)
    +(e^{2w}-e^{2u}\cos{2v})\eta(x)\eta(y),
\end{split}
\end{equation}
where $u$, $v$, $w$ are differentiable functions  on
$M$~ \cite{Man4}. These transformations form a group denoted by
$G$.
If $w=0$, we obtain the contactly conformal transformations of the
B-metric, introduced in  \cite{ManGri1}. By $v=w=0$, the
transformations \eqref{Transf} are reduced to the usual conformal
transformations of $g$.

Let us remark that $G$ can be considered as a contact complex
conformal gauge group, \ie the composition of an almost contact
group preserving $H$ and a complex conformal transformation of the
complex Riemannian metric $\overline{g^{\C}}=e^{2(u+iv)}g^{\C}$ on
$H$.

Note that the normality condition $N=0$ is not preserved by $G$.
In  \cite{ManIv38}, it is established that the tensor $N(\f\cdot,\f\cdot)$ is
an invariant of $G$ on any almost contact B-metric manifold and
 $\U_0$ is closed with respect to $G$.
By direct computations is established there  that each of  $\F_i$ $(i=1,2,\dots,11)$
 is closed by the action of the subgroup $G_0$ of $G$ defined by the conditions
$
    \D u\circ\f^2+\D v\circ\f=\D u\circ\f -\D v\circ\f^2 y=\D u(\xi)=\D
    v(\xi)=\D w\circ\f=0.
$ 
Moreover, $G_0$ is
the largest subgroup of $G$ preserving $\ta$, $\ta^*$,
$\om$ and  $\F_0$.
Moreover, the torsion of the
$\f$-canonical connection is invariant with respect to the general
contactly conformal transformations if and only if these
transformations belong to $G_0$ \cite{ManIv38}.

Bearing in mind the invariance of $\F_i$ $(i=1,2,\dots,11)$ and
$\ddot{T}$ with respect to the transformations of $G_0$,  each of the basic
classes of $\M$ is characterized by the torsion of
the $\f$-canonical connection
 as follows: \cite{ManIv38}
\[
\begin{array}{rl}
\F_1: &\ddot{T}(x,y)=\frac{1}{2n}\left\{\ddot{t}(\f^2 x)\f^2 y
-\ddot{t}(\f^2 y)\f^2 x
            +\ddot{t}(\f x)\f y-\ddot{t}(\f y)\f x\right\}; \\[0pt]
\F_2: &\ddot{T}(\xi,y)=0,\quad \eta\left(\ddot{T}(x,y)\right)=0,\quad
\ddot{T}(x,y)=\ddot{T}(\f x,\f y),\quad \ddot{t}=0;\\[0pt]
\F_3: &\ddot{T}(\xi,y)=0,\quad \eta\left(\ddot{T}(x,y)\right)=0,\quad
\ddot{T}(x,y)=\f \ddot{T}(x,\f y);\\[0pt]
\F_4: &\ddot{T}(x,y)=\frac{1}{2n}t'^*(\xi)\left\{\eta(y)\f x-\eta(x)\f y\right\};\\[0pt]
\F_5: &\ddot{T}(x,y)=\frac{1}{2n}t'(\xi)\left\{\eta(y)\f^2 x-\eta(x)\f^2 y\right\};\\[0pt]
\F_6: &\ddot{T}(x,y)=\eta(x)\ddot{T}(\xi,y)-\eta(y)\ddot{T}(\xi,x),\quad\\
&\ddot{T}(\xi,y,z)=\ddot{T}(\xi,z,y)=-\ddot{T}(\xi,\f y,\f z);\\[0pt]
\F_{7/8}: &\ddot{T}(x,y)=\eta(x)\ddot{T}(\xi,y)-\eta(y)\ddot{T}(\xi,x)
+\eta(\ddot{T}(x,y))\xi,\\[0pt]
            &\ddot{T}(\xi,y,z)=-\ddot{T}(\xi,z,y)=\mp \ddot{T}(\xi,\f y,\f z)\\
            &=\frac{1}{2}\ddot{T}(y,z,\xi)=\mp \frac{1}{2}\ddot{T}(\f y,\f z,\xi);\\[0pt]
\F_{9/10}: &\ddot{T}(x,y)=\eta(x)\ddot{T}(\xi,y)-\eta(y)\ddot{T}(\xi,x),
\quad\\ &\ddot{T}(\xi,y,z)=\pm \ddot{T}(\xi,z,y)=\ddot{T}(\xi,\f y,\f z);\\[0pt]
\F_{11}: &\ddot{T}(x,y)=\left\{\hat{\ddot{t}}(x)\eta(y)-\hat{\ddot{t}}(y)\eta(x)\right\}\xi.
\end{array}
\]

According to the classification of the torsions in  \cite{ManIv36}
and the characterization above, we have that the correspondence
between the classes $\F_i$ of $M$ and the classes $\T_{j}$ of the
torsion $\ddot{T}$ of the $\f$-canonical connection on $M=\M$ is
given as follows: \cite{ManIv38}
\[
\begin{array}{ll}
M\in\F_0\; \Leftrightarrow \;
\ddot{T}\in\T_{1}\oplus\T_{2}\oplus\T_{6}\oplus\T_{12}; \quad &
M\in\F_6\; \Leftrightarrow \; \ddot{T}\in\T_{11};
\\[0pt]
M\in\F_1\; \Leftrightarrow \; \ddot{T}\in\T_{4}; \quad &
M\in\F_7\; \Leftrightarrow \; \ddot{T}\in\T_{7}\oplus\T_{12};
\\[0pt]
M\in\F_2\; \Leftrightarrow \; \ddot{T}\in\T_{5}; \quad &
M\in\F_8\; \Leftrightarrow \; \ddot{T}\in\T_{8}\oplus\T_{14};
\\[0pt]
M\in\F_3\; \Leftrightarrow \; \ddot{T}\in\T_{3}; \quad &
M\in\F_9\; \Leftrightarrow \; \ddot{T}\in\T_{13};
\\[0pt]
M\in\F_4\; \Leftrightarrow \; \ddot{T}\in\T_{10}; \quad &
M\in\F_{10}\; \Leftrightarrow \; \ddot{T}\in\T_{14};
\\[0pt]
M\in\F_5\; \Leftrightarrow \; \ddot{T}\in\T_{9}; \quad &
M\in\F_{11}\; \Leftrightarrow \; \ddot{T}\in\T_{15}.
\end{array}
\]

\subsection{The $\f$KT-connection }

In  \cite{Man31}, it is introduced a natural connection $\dddot{D}$ on $\M$,
called a \emph{$\f$KT-connec\-tion}, which torsion $\dddot{T}$ is
totally skew-symmetric, \ie a 3-form. There it is proved that the
$\f$KT-connection exists on an almost contact B-metric manifold
 if and only if $\widehat{N}$ vanishes on it, \ie
when $\M\in \F_3\oplus\F_7$. The $\f$KT-connection is the
odd-dimensional analogue of the KT-connection introduced in
 \cite{Mek-2} on the corresponding class of quasi-K\"ahler manifolds with Norden metric.  The unique $\f$KT-connection
$\dddot{D}$ is determined by
    \[
            g(\dddot{D}_xy,z)=g(\n_xy,z)+\frac{1}{2}\dddot{T}(x,y,z),
    \]
    where the torsion is defined by
\begin{equation}\label{T37} %
\begin{split}
\dddot{T}(x,y,z)&=-\frac{1}{2} \sx\bigl\{F(x,y,\f z)-3\eta(x)F(y,\f
z,\xi)\bigr\}\\
&=\frac{1}{4}\sx
N(x,y,z)+\frac{1}{2}\left(\eta\wedge
\D\eta\right)(x,y,z). %
\end{split}
\end{equation} %
Obviously, the torsion forms of the $\f$KT-connection are zero.

The torsion $\dddot{T}$ of the $\f$KT-connection belongs to \(
\T_{3}\oplus\T_{6}\oplus\T_{7}\oplus\T_{12}\), according to \cite{ManIv36}.

From \eqref{T37} and \eqref{Fi:N}, for the classes  $\F_3$ and $\F_7$ we obtain
\begin{equation}\label{TKT-F3F7} %
\F_3:\quad \dddot{T}=\frac{1}{4} \sx
hN,\qquad
\F_7:\quad \dddot{T}=\eta\wedge
\D\eta. %
\end{equation} %

As mentioned above, the $\f$B-connection and the
$\f$-canonical con\-nec\-tion  coincide (\ie $\dot{D}\equiv\ddot{D}$) if and
only if $\M$ belongs to $\F_i$,
$i\in\{1,2,\dots,11\}\setminus \{3,7\}$ (where the
$\f$KT-connection $\dddot{D}$ does not exist).

For the rest basic classes $\F_3$ and $\F_7$ (where the $\f$KT-connection exists), according to \cite{ManIv36}, it is valid that
the $\f$B-con\-nec\-tion  is the \emph{average} connection of
the $\f$-canonical
con\-nec\-tion  and the $\f$KT-connection, \ie $\dot{D}=\frac12\left\{\ddot{D}+\dddot{D}\right\}$.
This relation  holds also because of \eqref{T0N-F3F7}, \eqref{Tcan-F3F7} and \eqref{TKT-F3F7}.



\begin{thebibliography}{33}

\bibitem{Blair}
D. E. Blair, Riemannian {G}eometry of {C}ontact and {S}ymplectic
{M}anifolds,  Progress in Mathematics, Vol. 203, Birkh\"auser,
Boston, 2002.

\bibitem{BoFeFrVo99}
A. Borowiec, M. Ferraris, M. Francaviglia, I. Volovich,
Almost complex and almost product einstein manifolds from a
variational principle, J. Math. Phys. 40 (1999) 3446--3464.

\bibitem{Car25} %
E. Cartan, Sur les vari\'et\'es \`a connexion affine et la
th\'eorie de la relativit\'e g\'en\'eralis\'ee (deuxi\`eme
partie), Ann. Ec. Norm. Sup. 42 (1925) 17--88, part II. English
transl. of both parts by A. Magnon and A. Ashtekar, On {M}anifolds
with an {A}ffine {C}onnection and the {T}heory of {G}eneral
{R}elativity, Bibliopolis, Napoli, 1986.


\bibitem{GaBo} 
G. Ganchev, A. Borisov, Note on the almost complex manifolds with
a Norden metric, C. R. Acad. Bulgare Sci.  39 (5)  (1986) 31--34.

\bibitem{GaGrMi85}
G. Ganchev, K. Gribachev, V. Mihova,
Holomorphic hypersurfaces of Kaehler manifolds with Norden
metric, Universite de Plovdiv ,,Paissi Hilendarski``, Travaux
scientifiques, Mathematiques 23 (2) (1985) 221--237;
arXiv:1211.2091.

\bibitem{GaGrMi87}
G. Ganchev, K. Gribachev, V. Mihova, B-connections and their
conformal invariants on conformally K\"ahler manifolds with
B-metric,  Publ. Inst. Math. (Beograd) (N.S.)  42 (1987) 107--121.

\bibitem{GaIv92}
G. Ganchev, S. Ivanov, Characteristic curvatures on complex
Riemannian manifolds, Riv. Mat. Univ. Parma (5)  1 (1992)
155--162.

\bibitem{GaMi87}
G. Ganchev, V. Mihova, Canonical connection and the canonical
conformal group on an almost complex manifold with B-metric, Ann.
Univ. Sofia Fac. Math. Inform. 81 (1987) 195--206.


\bibitem{GaMiGr}
G. Ganchev, V. Mihova, K. Gribachev, Almost contact manifolds with
B-metric, Math. Balkanica (N.S.)  7 (3-4) (1993) 261--276.


\bibitem{Gau97}
P. Gauduchon, Hermitian connections and Dirac operators, Boll.
Unione Mat. Ital. Sez. A Mat. Soc. Cult. (8) 11  (1997) 257--288.


\bibitem{GrMeDj} 
K. I. Gribachev, D. G. Mekerov, G. D. Djelepov, Generalized
B-manifolds,  C. R. Acad. Bulg. Sci. 38 (3)  (1985) 299--302.


\bibitem{IvMaMa14}
S. Ivanov, H. Manev, M. Manev, Sasaki-like
almost contact complex Rie\-mann\-ian manifolds,
arXiv:1402.5426.



\bibitem{LeB83} %
C. LeBrun, Spaces of complex null geodesics in complex-Riemannian
geometry, Trans. Amer. Math. Soc. 278 (1) (1983) 209--231.

\bibitem{Lib54}
P. Libermann, Sur les connexions hermitiennes, C. R. Acad. Sci.
Paris 239  (1954) 1579--1581.


\bibitem{Lih62}
A. Lichnerowicz, {T}h\'{e}orie {G}lobale des {C}onnections et des
{G}roupes d'{H}o\-mo\-to\-pie, Edizioni Cremonese, Roma, 1962.


\bibitem{Man3}
M. Manev, Properties of curvature tensors on almost contact
manifolds with B-metric, In:  Proc. of Jubilee Sci. Session of
Vassil Levsky Higher Mil. School , Vol. 27, Veliko Tarnovo,
Bulgaria, 1993, pp.  221--227.

\bibitem{Man4}
M. Manev, Contactly conformal transformations of general type of
almost contact manifolds with B-metric. Applications, Math.
Balkanica (N.S.) 11 (3-4) (1997) 347--357.

\bibitem{Man8}
M. Manev, Almost contact B-metric hypersurfaces of
Kaeh\-ler\-ian mani\-folds with B-met\-ric, In: Perspectives of
Complex ana\-lysis, Differential Geo\-metry and Mathematical
Physics, Eds. St. Dimiev and K. Sekigawa, World Sci. Publ.,
Singapore, 2001, pp. 159--170.

\bibitem{Man31}
M. Manev, A connection with totally skew-symmetric torsion on
almost con\-tact manifolds with B-metric, Int. J. Ge\-om. Methods
Mod. Phys.  9 (5) (2012) 1250044 (20 pa\-ges).


\bibitem{ManGri1}
M. Manev, K. Gribachev, Contactly conformal transformations of
almost contact manifolds with B-metric, Serdica Math. J. 19
 (1993) 287--299.


\bibitem{ManGri2}
M. Manev, K. Gribachev, Conformally invariant tensors on almost
con\-tact manifolds with B-metric, Serdica Math. J. 20 (1994)
133--147.

\bibitem{ManIv37}
M. Manev, M. Ivanova, A natural connection on some classes of
al\-most contact manifolds with B-metric, C. R. Acad. Bulg. Sci.
65 (4)  (2012) 429--436.

\bibitem{ManIv38}
M. Manev, M. Ivanova, Canonical-type connection on al\-most
contact manifolds with B-metric, Ann. Global Anal. Ge\-om. 43 (4)
  (2013) 397--408.

\bibitem{ManIv36}
M. Manev, M. Ivanova, A classification of the torsion tensors on
almost contact manifolds with B-metric, Cent. Eur. J. Math.  12 (10) (2014) 1416--1432.

\bibitem{Manin88}
Y. I. Manin, Gauge field theory and complex geometry.
Translated from the Russian by N. Koblitz and J. R. King.
Grundlehren der Mathematischen Wissenschaften [Fundamental
Principles of Mathematical Sciences], 289. Springer-Verlag,
Berlin, 1988. x+297 pp.

\bibitem{Mek-1}
D. Mekerov, On the geometry of the B-connection on
quasi-K\"{a}hler manifolds with Norden metric, C. R. Acad. Bulg.
Sci.   61   (2008) 1105--1110.

\bibitem{Mek-2}
D. Mekerov, A connection with skew symmetric torsion and
K\"{a}hler cur\-vature tensor on quasi-K\"{a}hler manifolds with
Norden metric, C. R. Acad. Bulg. Sci. 61 (2008) 1249--1256.

\bibitem{Mek09}
D. Mekerov, Canonical connection on quasi-K\"ahler manifolds with
Norden metric, J. Tech. Univ. Plovdiv Fundam. Sci. Appl. Ser. A
Pure Appl. Math. 14  (2009) 73--86.

\bibitem{MeMa}
D. Mekerov, M. Manev, On the geometry of quasi-K\"{a}hler
manifolds with Norden metric, Nihonkai Math. J. 16 (2) (2005)
89--93.

\bibitem{NakGri2}
G. Nakova, K. Gribachev, Submanifolds of some almost
contact manifolds with B-metric with codimension two, I, Math.
Balkanica 11 (1997) 255--267.

\bibitem{Nor60}
A. P. Norden, On a class four-dimensional A-spaces, Izv.
VUZ, Matematika 17 (1960) 145--157. (in Russian)

\bibitem{Nor72}
A. P. Norden, On the structure of connections on the space
of lines on non-Euclidean spaces, Izv. VUZ, Matematika 127
(1972) 82--94. (in Russian)


\bibitem{Sal89}
S. Salamon, {R}iemannian {G}eometry and {H}olonomy {G}roups,
Pitman Research Notes in Mathematical Series, Vol. 201, Jon Wiley
\& Sons, 1989.

\bibitem{SaHa}
S. Sasaki, Y. Hatakeyama, On differentiable manifolds with certain
structures which are closely related to almost contact structures
II, T\^{o}hoku Math. J. 13 (1961) 281--294.

\bibitem{TV83}
F. Tricerri, L. Vanhecke, Homogeneous structures on Riemannian
manifolds, in  London {M}ath. {S}oc. {L}ecture {N}otes {S}eries
Vol. 83, Cambridge Univ. Press, Cambridge, 1983.


\end{thebibliography}
\end{document}